\documentclass[12pt,a4paper,fleqn]{article}

\usepackage{amssymb,amsmath}

\newtheorem{theorem}{Theorem}[section]
\newtheorem{lemma}[theorem]{Lemma}

\newcommand{\ord}{{\rm ord}}

\newcommand{\Li}{{\rm Li}}

\newenvironment{proof}{{\it Proof}:}{\hfill$\square$\\ $\phantom{A}$\\}

\begin{document}

\title{Integration in terms of polylogarithm}

\author{
W. Hebisch \\
{\small  Mathematical Institute, Wroc\l aw University, Poland}
}
\date{}
\maketitle

\begin{abstract}
This paper provides a Liouville principle for integration in
terms of dilogarithm and partial result for polylogarithm. 
\end{abstract}

\section{Introduction}

Indefinite integration is classical task studied from beginning
of calculus: given $f$ we seek $g$ such that $f = g'$.
The first step in deeper study of integration is
to delimit possible form of integrals.  In case of elementary
integration classical Liouville-Ostrowski theorem says that
only new transcendentals that can appear in $g$ are logarithms.
More precisely,
when $f\in L$ where $L$ is a differential field with
algebraically closed constant field and $f$ has integral
elementary over $L$, then
$$
f = v_0' + \sum c_i\frac{v_i'}{v_i}
$$
where $v_i \in L$ and $c_i\in L$ are constants.

However, there are many elementary function which do not have
elementary integrals and to integrate them we introduce new
special functions in the integral.  In this paper we study
integration in terms of polylogarithms.  Polylogarithms
appear during iterated integration of rational functions
\cite{Kummer}, so they are very natural extension of
elementary functions.  Integration
in terms of polylogarithms was studied by Baddoura \cite{Bad94}, \cite{Bad06},
but he only handled integrals in transcendental extensions.
Also, he gave proofs only in case of dilogarithm.
In recent article Y. Kaur and V. R. Srinivasan \cite{KaSri}
give Liouville type principle for larger class of functions.
They use different arguments, but for dilogarithm their
results are essentially equivalent to that of Baddoura.
  When
seeking integral we allow also algebraic extensions.
We give partial result for polylogarithm (Theorem \ref{exp-case}).
It seems that we give first proof of nontrivial symbolic
integration result for polylogarithms of arbitrary integer
order (Baddoura in \cite{Bad11} gives a useful result, but
leaves main difficulty unresolved).

In case of primitive extensions our result are based on abstract
version of dilogarithm identity given by Baddoura.  In exponential
case we base our proof on a lemma about independence of
logarithmic forms (Lemma \ref{pole-indep}) which may be of independent
interest.

\section{Setup and preliminaries}

Classical polylogarithm $\Li_s$ is defined by series
$$
\Li_s(z) = \sum_{k=1}^\infty \frac{z^k}{k^s}
$$
which is convergent for $|z| < 1$ and extended by
analytic continuation to multivalued function.
For purpose of symbolic integration of particular interest
are polylogarithms of integer order.  Differentiating
the series we get the reccurence relation
$$
z\partial_z\Li_s(z) = \Li_{s-1}(z)
$$
so
$$
\Li_s(z) = \int_{0}^z \frac{\Li_{s-1}(t)}{t}dt.
$$
We have
$$
\Li_1(z) = -\log(1 - z)
$$
so for positive integer $n$ polylogarithm $\Li_n(z)$ is
a Liouvillian function.  For $n > 2$ multiple integration
implied by the reccurence relation is inconvenient, so
we introduce functions $I_m$ by the formula
$$
I_m(z) = \int_0^z \log(t)^{m - 1}\frac{dt}{1 - t}.
$$
For integer $m \geq 1$ we have
$$
\Li_m(z) = \frac{(-1)^{m-1}}{(m - 1)!}I_m
-\sum_{k=1}^{m-1} \frac{(-1)^k}{k!}\Li_{m - k}(z)\log(z)^k.
$$
Namely, applying $z\partial_z$ to both sides, by direct calculation
we get equality, so difference between both sides is a constant.
By computing limit when $z$ tends to $0$ we see that the constant
is zero.
Consequently, $\Li_m(z)$ can be expressed in terms of $\log(z)$
and $I_k(z)$, with $k=1,\dots,m$.  Similarly $I_m(z)$ can be expressed
in terms of $\log(z)$ and $\Li_k(z)$, with $k=1,\dots,m$.
This means that integration in terms of polylogarithms is
equivalent to integration in terms of $I_m(x)$.

In the sequel we assume standard machinery of differential fields
(see for example \cite{Ros}).  We will denote derivative in
differential fields by $D$, except for cases when we will need
more than one derivative.
When $u$
is element of a differential field $K$ we
have $DI_k(u) = (D(u-1)/(u-1))\log(u)^k = (D\log(u - 1))\log(u)^k$.

We say that a differential field $L$
is a dilogarithmic extension of $F$ iff there exists
$\theta_1, \dots, \theta_n \in L$ such that
$L = F(\theta_1,\dots, \theta_2)$ and for each $i$, $1\leq i \leq n$,
one of the following holds
\begin{enumerate}
\item $\theta_i$ is algebraic over $F_i$
\item $\frac{D\theta_i}{\theta_i} = Du$ for some $u \in F_i$
\item $D\theta_i = \frac{Du}{u}$ for some $u \in F_i$
\item $D\theta_i = \frac{D(u - 1)}{u - 1}v^k$ for some $u, v \in F_i$
such that $Du = (Dv)u$ and integer $k > 0$
\end{enumerate}
where $F_i = F(\theta_1, \dots, \theta_{i-i})$.  When only
first three cases appear we say that $L$ is an elementary
extension of $F$.  Intuitively
clauses 2 to 4 above mean $\theta_i = \exp(u)$,
$\theta_i = \log(u)$, $\theta_i = I_k(u)$ respectively.
However, logarithm and $I_k$ is only determined up to additive constants
by equations above.  Similarly, exponential is only determined
up to multiplicative constants.  Our results about integrability
do not depend on specific choice of exponentials and logarithms:
different choice only changes elementary part of the integral,
but does not affect integrability.  We take advantage of
this freedom and in the proofs assume that for nonzero $x$ and $y$ we have
\begin{equation}\label{log-add}
\log(xy) = \log(x) + \log(y),
\end{equation}
\begin{equation}\label{log-neg}
\log(-x) = \log(x).
\end{equation}
Note that the second formula implies that $\log(-1) = 0$, which
does not agree with definition in calculus, but still satisfies
$D\log(-1) = D(-1)/(-1) = 0$.  In fact, we assume that logarithms
of roots of unity are all $0$.  Note that torsion subgroup of
multiplicative group $F_{*}$ of a field $F$ consists of roots of unity.  So
$F_{*}$ divided by roots of unity is a torsion free group.
Any finitely generated subgroup of an abelian torsion free group
is a free subgroup.  So given any finite subset $S$ of $F_{*}$
we consider multiplicative subgroup $G$ generated by $S$ modulo
roots of unity.  Then on generators $g$ of $G$ we can choose
logarithms in any way consistent with equation $D\log(g) = (Dg)/g$
and extend by linearity to whole $G$.  This ensures that
(\ref{log-add}) and (\ref{log-neg}) hold on subgroup of $F_{*}$
generated by $S$.  Since we will simultaneously use only finite
number of elements, we can assume that (\ref{log-add}) and
(\ref{log-neg}) hold for all elements that we will use.  Also,
in various places when we consider logarithms we assume that
they will not add new constants.  Namely, adding transcendental
$\log(a)$ to a base field adds new constant iff $a$ already has
logarithm in the base field.  So by adding transcendental
logarithms only when we can not find logarithm in base field
we ensure that there will be no new constants.

It is well-known that when a field $K$ is an extension of transcendental
degree $1$ of $F$, which is finitely generated over $F$, then
$K$ can be treated as a function field on an algebraic curve
defined over $F$ (\cite{Chev}).  Standard tool in this situation is
use of Puiseaux expansions.

For convenience we give a few known lemmas:
\begin{lemma}\label{mult-free}
Let $F$ be a field, $\bar F$ its algebraic closure,
$K$ an extension of $F$ of transcendental
degree $1$, $\psi_1,\dots,\psi_n \in K - \{0\}$ a finite
family.  Multiplicative group $G$ generated by $\psi_1,\dots,\psi_n$
modulo $\bar F$ is a free abelian group.
\end{lemma}
\begin{proof}
Without loss of generality we may assume that $K$ is generated
by $\psi_1,\dots,\psi_n$ over $\bar F$, so $K$ is a finite extension
of $\bar F$
Since $G$ is finitely generated abelian group it is enough to
show that it is torsion free.  However, a nonzero element of
$G$ is a algebraic function $f$ not in $\bar F$, that is having
a zero at same place $p$.  Any power of $f$ has zero at $p$,
so is not in $\bar F$, so not a zero element of $G$.
\end{proof}

\begin{lemma}\label{log-lin}
Let $F$ be a differential field, $v_1, \dots, v_n \in F - \{0\}$.
Assume that $F$ has the same constants as $F(\log(v_1), \dots, \log(v_n))$.
Then
$\log(v_1), \dots, \log(v_n)$ are algebraically dependent over
$F$ if and only if $\log(v_i)$ are linearly dependent over constants
modulo $F$.
\end{lemma}
\begin{proof}
We have $n$ equations:
$$
\frac{Dv_i}{v_i} - D\log(v_i) = 0.
$$
If $\log(v_1), \dots, \log(v_n)$ are algebraically dependent,
then transcendental degree of $K = F(\log(v_1), \dots, \log(v_n))$
over $F$ is smaller than $n$.
By \cite{Ros} Theorem 1 differential forms
$$
\frac{dv_i}{v_i} - d\log(v_i) \in \Omega_{K/F}
$$
are linearly dependent over constants, so there exists
nonzero constants $c_1,\dots,c_n$ such that
$$
\sum c_i\frac{dv_i}{v_i} - d\left(\sum c_i\log(v_i)\right) = 0.
$$
Choose basis
$\gamma_1,\dots,\gamma_r$ for vector space over rationals
generated by $c_1,\dots,c_n$ so that each $c_i$
can be written as $c_i = \sum \alpha_{i,j}\gamma_j$
with integer $\alpha_{i,j}$.  Write
$t_j = v_1^{\alpha_{1,j}}\dots v_n^{\alpha_{n,j}}$.
Then
$$
\sum \gamma_j\frac{dt_j}{t_j} + dw = 0
$$
where $w = \sum c_i\log(v_i)$.  By \cite{Ros} Proposition 4
$w$ is algebraic over $F$.  Since
$$
Dw = \sum c_i\frac{Dv_i}{v_i}
$$
and right hand side is in $F$ we have $Dw \in F$.  By taking
trace we see that $w$ is in $F$.  So $\log(v_i)$ are linearly
dependent over constants modulo $F$.
\end{proof}

\begin{lemma}\label{log-int}
Let $F$ be a differential field with algebraically closed
constant field, $v_1, \dots, v_n \in F - \{0\}$.
Assume that $F$ has the same constants as $F(\log(v_1), \dots, \log(v_n))$.
If $f \in F[\log(v_1), \dots, \log(v_n)]$
has integral elementary over $F$, then there exists
$E \in F[\log(v_1), \dots, \log(v_u)]$,
$u_1,\dots,u_k \in F$  and constants $c_1,\dots,c_k$ such that
$$
f = DE + \sum c_i\frac{Du_i}{u_i}.
$$
\end{lemma}
\begin{proof}
Without loss of generality we may assume that $\log(v_1),\dots,
\log(v_k)$ are algebraically independent and
$\log(v_{k+1}),\dots,\log(v_n)$ are algebraic over
$K = F(\log(v_1), \dots, \log(v_k))$.  By Lemma \ref{log-lin}
for $l = k+1,\dots,n$ we can write $\log(v_l)$ as
a linear combination of $\log(v_1), \dots, \log(v_k)$ and element
from $F$.  So, we can rewrite $f$ in terms of $\log(v_1),\dots,\log(v_k)$,
that is assume that $f \in F[\log(v_1), \dots, \log(v_k)]$.
Now, we can use \cite{Bro:str} Theorem 1.  Namely, 
due to algebraic independence of $\log(v_1), \dots, \log(v_k)$ 
the ring $R = F[\log(v_1), \dots, \log(v_k)]$ is a polynomial
ring in $k$ variables and $D$ is a derivation on $R$. 
Since $Dp$ has lower degree than $p$ by \cite{Bro:str} Proposition 1
there are no irreducible special polynomials.  Since $DF \subset F$
we can apply \cite{Bro:str} Theorem 1.  Our $f$ has denominator $1$,
so by \cite{Bro:str} Theorem 1 also $E$ has denominator $1$
and in logarithmic part we get logarithms of elements of $F$
(here we use fact that there are no irreducible special polynomials).
\end{proof}

\section{Main results}

Let $K$ be a differential field and $F$ be a differential
subfield of $K$.  We say that $f \in K$
has integral with polylog terms defined over $F$ when
there is extension $L$ of $K$ by some number of logarithms
of elements of $K$ and in $L$ we have
$$
f = DE + \sum d_i\frac{D(1 - h_i)}{1 - h_i}\log(h_i)^{k_i}
$$
where $h_i \in F$, $d_i$ are constants $k_i$ are
positive integers and $E$ is elementary over $K$.
We say that $f \in K$ has integral with dilog terms defined over $F$ when
there is expression as above and all $k_i = 1$.

Note that applying Lemma \ref{log-int} to equation
$g = DE$ where $g$ is difference of $f$ and polylog terms
we see that $E$ is is sum of polynomial
in $\log(h_i)$ with coefficients in $K$ and linear
combination of logarithms of elements of $K$ with
constant coefficients.

\begin{theorem}\label{exp-case}
Let $f\in F$, $f$ has integral with polylog terms defined
over $K$.
If $\theta$ is an exponential,  $K$ is algebraic over $F(\theta)$,
$F$ and $K$ have the same constants,
then $f$ has integral with polylog terms defined over an
algebraic extension of $F$.
\end{theorem}

\begin{theorem}\label{prim-case}
Let $f\in K$, $f$ has integral with dilog terms defined
over $K$.  If $\theta$ is a primitive, $K$ is algebraic over $F(\theta)$,
$F$ and $K$ have the same constants,
then $f$ has integral with dilog terms defined over an
algebraic extension of $F$.
\end{theorem}

\begin{theorem}
Let $f\in F$, $f$ has integral in a dilogarithmic extension of
$F$ then $f$ has integral with dilogarithmic terms defined over an
algebraic extension of $F$.
\end{theorem}

\begin{proof}
Let $L$ be a dilogarithmic extension such that $f$ has integral in $L$.
First note that we can assume that $L$ has the same constants
as $F$.  Namely, $L = F(\theta_1, \dots, \theta_n)$
where each $\theta_i$ satisfies differential equation over
$F(\theta_1, \dots, \theta_{i-1})$ (note that algebraic
equation is treated as differential equation of order $0$).
In other words,
$L = F(\theta_1, \dots, \theta_n)$ is a dilogarithmic extension of
$F$ if and only if $\theta_1, \dots, \theta_n$ satisfy
appropriate system of differential equations.  $f = \gamma'$
is also a differential equation.  Like in Lemma 2.1 \cite{SSC}
clearing denominators
we convert system of differential equations to a differential
ideal $I$ plus an inequality $g \ne 0$ (which is responsible
for non-vanishing of denominators) and use result of Kolchin
which says that differential ideal $I$ which has zero in some
extension satisfying $g \ne 0$ has zero in extension having
constants algebraic over constants of $F$ and satisfying
$g \ne 0$.  Solution clearly gives us dilogarithmic extension
with constants algebraic over constants of $F$ such that
$f = \gamma'$ has solution.  Adding new constants to $F$
we get algebraic extension which does not affect our claim.
So we can assume that $L$ and $F$ have the same constants.

By definition of dilogarithmic extension there exits tower
$F = F_0 \subset F_1 \subset \dots \subset F_n = L$ such that
$F_{k+1}$ is algebraic over $F_k(\eta_k)$, each $\eta_k$
is either primitive over $F_k$ or an exponential over $F_k$.
By assumption $f$ has integral with dilogarithmic terms defined over
$F_n$.  Using induction and theorems \ref{exp-case} and \ref{prim-case}
we see that $f$ has integral with dilogarithmic terms defined over $F$.
\end{proof}

\section{Extension by exponential}

\begin{lemma}\label{pole-indep0}
Let $F$ be a differential field, $K$ be differential
field algebraic over $F(\theta)$, $\theta$ be an
exponential of a primitive or a primitive over $F$.  Assume that
$F$ is algebraically closed in $K$, $F$ and $K$ have the same constants,
$\theta$ is transcendental and $\psi_i \in K$,
$i=1,\dots,n$ are such that $\psi_i$
are
multiplicatively independent modulo $F_{*}$.
When $\theta$ is exponential of a primitive we assume that
$\theta$ and $\psi_i$ are
multiplicatively independent modulo $F_{*}$.
If $a_i \in F$, $Da_i = 0$, $s\in K$,
$$
\sum a_i \frac{D\psi_i}{\psi_i} + Ds \in F,
$$
then $a_i = 0$ for all $i$.
\end{lemma}

\begin{proof}
The result follows from results of Rosenlicht \cite{Ros}.
Namely, when $\theta$ is a primitive we
we have two equations, one from assumption and
the second $D\theta \in F$ from definition of $\theta$,
but the transcendental degree is one, so by Theorem 1
of \cite{Ros} there is differential form
$$
\sum c_i \frac{d\psi_i}{\psi_i} + dv = 0
$$
with constant $c_i$ and $v\in K$.  Choose basis
$\gamma_1,\dots,\gamma_r$ for vector space over rationals
generated by $c_1,\dots,c_n$ so that each $c_i$
can be written as $c_i = \sum \alpha_{i,j}\gamma_j$
with integer $\alpha_{i,j}$.  Write
$t_j = \psi_1^{\alpha_{1,j}}\dots\psi_n^{\alpha_{n,j}}$.
Then
$$
\sum \gamma_j\frac{dt_j}{t_j} + dv = 0.
$$
By \cite{Ros} Proposition 4 all $t_j$ are algebraic
over $F$.  Since $t_j \in K$ and $F$ is algebraically closed in $K$
we have $t_j \in F$, which means that $\psi_i$ are multiplicatively
dependent modulo $F_{*}$.

When $\theta$ is exponential of a primitive argument is similar,
but we need to include $\theta$  in the dependence.
\end{proof}

\begin{lemma}\label{pole-indep}
Let $F$ be a differential field, $K$ be differential
field algebraic over $F(\theta)$, $\theta$ be an
exponential of a primitive or a primitive over $F$.  Assume that
$F$ is algebraically closed in $K$, $F$ and $K$ have the same constants,
$\theta$ is transcendental and $\psi_i \in K$,
$i=1,\dots,n$ are such that $\psi_i$
are
multiplicatively independent modulo $F_{*}$.
When $\theta$ is exponential of a primitive we assume that
$\theta$ and $\psi_i$ are
multiplicatively independent modulo $F_{*}$.
If $a_i \in F$, $s\in K$,
$$
\sum a_i \frac{D\psi_i}{\psi_i} + Ds \in k,
$$
then $a_i = 0$ for all $i$.
\end{lemma}
Remark: This differs from Lemma \ref{pole-indep0} because
we dropped assumption that $a_i$ are constants.

\begin{proof}  We may assume that $K$ is finitely generated over $F(\theta)$.
When derivation on $F$ is trivial the result is just Lemma \ref{pole-indep0}.

To handle general derivative on $F$
first note that if $s$ has pole in normal place of ramification
index $r$, than $Ds$ has order less than $-r$, while
$\frac{D\psi_i}{\psi_i}$ has order at least $-r$
(see \cite{Bro:ele}, Lemma 1.7).
This means that pole of $Ds$ and poles of $\frac{D\psi_i}{\psi_i}$
can not cancel.  Consequently, since sum is in $F$ we see
that $s$ has no normal poles.

$\frac{D\psi_i}{\psi_i}$ is regular at special
places (see \cite{Bro:ele}, Lemma 1.8).
If $\theta$ is an exponential of a primitive, and $s$ had
pole at a special place, then also $Ds$ would have
pole at special case, which is impossible since
the sum is in $F$.  In other words, when $\theta$ is
an exponential of a primitive, then $s$ has no poles, so $s$ is
algebraic over $F$.  But $F$ is algebraically closed in
$K$ so $s \in F$.

Consider now mapping which maps $\psi_i$ to vector of
multiplicities of zeros and poles at normal places.
This mapping extends by linearity to mapping $\iota$ on
vector space over ${\mathbb Q}$ spanned by $\frac{D\psi_i}{\psi_i}$.
If $a_i\in {\mathbb Z}$,
$$
\iota(\sum a_i \frac{D\psi_i}{\psi_i})=0,
$$
then $\psi_i^{a_i}$ is a function with no normal poles or zeros.
By replacing $\psi_i$ by appropriate power products without
loss of generality we can assume that $\iota(\frac{D\psi_i}{\psi_i})$ for 
$i=1,\dots,l$ are linearly independent and
$\iota(\frac{D\psi_i}{\psi_i}) = 0$
for $i=l+1,\dots, n$.  $\iota$ takes values in ${\mathbb Q}^A$
where $A$ is set of normal zeros and poles of $\{\psi_i\}$.
We can extend $\iota$ by linearity to mapping from linear combinations
of $\frac{D\psi_i}{\psi_i}$ with coefficients in $F$ into
$F^A$.  Of course
$\iota(\frac{D\psi_i}{\psi_i})$ for $i=1,\dots,l$ remain linearly independent
over $F$.  However, we can compute $\iota$ from
coefficients of Puiseaux expansions of
$$
\sum a_i \frac{D\psi_i}{\psi_i}
$$
at places in $A$.  Namely, coefficient of order $-r$,
where $r$ is ramification index of the place is
order of zero of $\psi_i$ times $a_i$.  Moreover,
since $s$ has no normal poles, $Ds$ has order
bigger than $-r$, so
$$
\sum a_i \frac{D\psi_i}{\psi_i} + Ds \in F
$$
means that
$$
\iota(\sum a_i \frac{D\psi_i}{\psi_i}) = 0
$$
so $a_i=0$ for $i=1,\dots, l$.  In other words, to prove the
lemma it remains to handle case when all $\psi_i$ have no
normal zeros or poles.

Consider now Puiseaux expansion of $\frac{D\psi_i}{\psi_i}$
at a normal place $p$ of ramification index $r$.   Denoting by $\lambda$ parameter of expansion
we have
$$
\psi_i = c_0 + c_1\lambda^{1/r} + c_2\lambda^{2/r} + \dots
$$
$$
D\psi_i = D(\lambda)\partial_\lambda\psi_i + 
D(c_0) + D(c_1)\lambda^{1/r} + \dots.
$$
The second part contains only terms of nonnegative oder, so
$D\psi_i$ has the same terms of negative oder as
$D(\lambda)\partial_\lambda\psi_i$.  Since $s$ has nonnegative
order at $p$ the same argument shows that terms of negative order in $Ds$
are the same as terms of negative order of $D(\lambda)\partial_\lambda s$.
  At normal place $p$
$D(\lambda)$ has order $0$ so negative part of Puiseaux expansion
at $p$ of
$$
\sum a_i \frac{D\psi_i}{\psi_i} + Ds
$$
is the same as of
$$
D(\lambda) \left( \sum a_i \frac{\partial_\lambda\psi_i}{\psi_i}
+ \partial_\lambda s\right).
$$
In particular terms of negative order vanish if and only if
terms of negative order in expansion of
$$
\sum a_i \frac{\partial_\lambda\psi_i}{\psi_i}
+ \partial_\lambda s
$$
vanish.  But the last expression does not depend on derivative
$D$.  Let $X$ be derivative on $K$ such that $X$ is zero on
$F$ and $X\theta = D\theta$.  Note that for $D$ and $X$
we have the same set of special places.
By reasoning above,
$$
\sum a_i \frac{D\psi_i}{\psi_i} + Ds \in F
$$
implies that
$$
\sum a_i \frac{X\psi_i}{\psi_i} + Xs
$$
has nonnegative order when $p$ is normal place.
For special places $\frac{X\psi_i}{\psi_i}$ has nonnegative
order.  When $\theta$ is exponential of a primitive we observed
that $s \in k$, so $Xs = 0$ and consequently also has nonnegative order.
So in all places
$$
\sum a_i \frac{X\psi_i}{\psi_i} + Xs
$$
has nonnegative order, so is algebraic over $F$ so in $F$, which by the
Lemma \ref{pole-indep0} means that $a_i = 0$.

It remains to handle case when $\theta$ is a primitive.  Then
we have $D\theta = \eta \in F$ and we can use $\theta^{-1/r}$ as parameter
in Puiseaux expansion at special places.  We have
$$
s = \sum c_i\theta^{-i/r}
$$
$$
Ds = \sum Dc_i\theta^{-i/r} - \sum \frac{i}{r}c_i\eta\theta^{-(i+r)/r}
$$
$$
= \sum Dc_i\theta^{-i/r} - \sum \frac{i-r}{r}c_{i-r}\eta\theta^{-i/r}
$$
$Ds$ has nonnegative order at special places, so all terms above
with negative $i$ vanish.  When $i<0$ is lowest order term such
that $c_i \ne 0$, then $Dc_i\theta^{-i/r}$ can not cancel with other
terms so $Dc_i = 0$.  Now, $i<-r$ implies that
$$
Dc_{i+r} - \frac{i}{r}c_i\eta = 0
$$
that is
$$
D\theta = \eta = D\frac{-c_{i+r}}{c_i}
$$
but this is impossible since $\theta$ is transcendental and $F$ and
$K$ have the same constants.  So $-r \leq i < 0$.
Now, similarly like $Dc_i$ we see that $Dc_j = 0$ for $j<0$.
We can do the same calculation for $Xs$ and we see that
all terms of negative order in $Xs$ vanish.  So, we can finish like in
exponential case.
\end{proof}



Now we are ready to prove Theorem \ref{exp-case}.

\begin{proof}
When $\theta$ is algebraic over $F$ there is nothing to prove,
so we may assume that $\theta$ is transcendental.
Let $\bar F$ be algebraic closure of $F$.  By assumption we have
$$
f = g_F + DE + \sum d_i\frac{D(1 - h_i)}{1 - h_i}\log(h_i)^{k_i}
$$
where $g_F$ is sum of polylog terms with arguments in $\bar F$,
$E$ denotes elementary part and $h_i \in K - \bar F$ are arguments of
polylogs outside $\bar F$.  Consider group $G$ generated
by arguments of
logarithms in elementary part, $h_i$ and $1 - h_i$ modulo $\bar F$.  By
Lemma \ref{mult-free} $G$ is a free abelian group.
Let $\alpha_j \in K$ be generators of $G$.  In particular they
are multiplicatively independent over $\bar F$.
Let $p$ be a place
of $K$ over $0$.  We may assume that each $\alpha_j = \theta^k\beta_j$
where $\beta_j$ has nonzero value at $p$ and $k$ depend on $j$.
Namely, replacing $\theta$
by a fractional power we may assume that order of $\theta$ at $p$
divides orders of all $\alpha_j$.  
Next, we normalize $\beta_j$
so that each has value $1$ at $p$ and chose a multiplicatively independent
family $\psi_j$ which generate the same subgroup of $(\bar FK)_{*}$ as
$\beta_j$ (again, we can do this due to Lemma \ref{mult-free}).
Note that due to Lemma \ref{pole-indep0}  $\log(\psi_j)$ are linearly
independent over constants modulo $\bar F$ so by
and \ref{log-lin} they
are algebraically independent over $K$.
Then each $h_i$ can be written
as
$$
h_i = u_i\prod \psi_j^{n_{i,j}}
$$
where
$n_{i,j}$ are integers and
$u_i = w_i\theta^{l_i}$ with $w_i$ algebraic over $F$ and integer $l_j$
so
$$
\log(h_i) = \sum n_{i,j}\log(\psi_j) + \log(u_i).
$$
Similarly
$$
\log(1 - h_i) = \sum m_{i,j}\log(\psi_j) + \log(v_i)
$$
where $m_{i,j}$ are integers and $v_i = r_i\theta^{o_i}$
with $r_i$ algebraic over $F$ and integer $o_i$.
After rewriting $h_i$, $\log(1 - h_i)$ as above from polylog
terms we get polynomial in $\log(\psi_j)$, with coefficients
in $N = F(\Delta \cup \{\eta\}\cup \{w_i, \log(w_i), r_i, \log(r_i)\})$ where
$\Delta$ is set of logarithms needed to express $g_F$
and $\eta = \frac{D\theta}{\theta}$.
We also rewrite logarithms in elementary
part like above.  By Lemma \ref{log-int} we see that
now $E$ is a polynomial in $\log(\psi_j)$, $j>0$ with coefficients in 
$M = NK$.  Put $l_\alpha = \prod\log(\psi_j)^{\alpha_j}$.
We have
$$
E = \sum s_\alpha l_\alpha
$$
with $s_\alpha \in M$.
Expanding formula for $f$
in terms of $l_\alpha$ we get system of equations
$$
\sum c_{\alpha, j}\frac{D(\psi_j)}{\psi_j} + Ds_\alpha + c_{\alpha, 0}= 0
$$
where $c_{\alpha, j}$ are in $M$.
Now, $M$ is algebraic over $N(\theta)$.  We claim that
$c_{\alpha, j}\in {\bar N}$ and $s_\alpha\in {\bar N}$
where ${\bar N}$ is algebraic closure of $N$ in $M$.
We prove this inductively.  The claim
is vacuously true if length of $\alpha$ is big enough
so that $c_{\alpha, i} = 0$ and $s_\alpha=0$.
So we may assume that our claim is true for all
multiindices of higher length.
Note that for $j>0$ $c_{\alpha, j}$ is sum of $s_{\alpha + e_j}$ and
terms coming from polylogs.  $c_{\alpha, 0}$ is sum of terms coming
from polylogs and for $\alpha = 0$ also includes $g_F - f$.
Of course terms coming from polylogs and $g_F - f$ are in $N$.
Consequently by the inductive assumption
$c_{\alpha, j} \in \bar N$.
Hence, by the lemma \ref{pole-indep} for $j>0$ we
have $c_{\alpha, j} = 0$, so $s_\alpha$ has derivative in
$\bar N$.  Consequently, since $\theta$ is an exponential we
have $s_\alpha \in \bar N$.

Now, we look at equality with $\alpha = 0$.  From derivative
of $E$ we get
$Ds_0 + \sum_{j>0} s_{e_j}\frac{D(\psi_j)}{\psi_j}$.  From
$$d_i\frac{D(1- h_i)}{1 - h_i}\log(h_i)^{k_i}$$
we get
$$
d_i\left(\frac{D(v_i)}{v_i} + \sum m_{i,j}\frac{D(\psi_j)}{\psi_j}\right)
\log(u_i)^{k_i}
$$
so
$$
f = g_F + \sum d_i\frac{D(v_i)}{v_i} \log(u_i)^{k_i} + Ds_0 +
\sum c_{0, j}\frac{D(\psi_j)}{\psi_j}.$$
We proved above that $c_{0, j} = 0$ for $j>0$, so 
this simplifies to
$$
f = g_F + \sum d_i\frac{D(v_i)}{v_i}\log(u_i)^{k_i} + Ds_0.
$$

We look at $u_i$ and $v_i$.  First, if $u_i$ or $v_i$ have a pole
at $p$, then $v_i = -u_i$ and we chose logarithm in such a
way that $\log(v_i) = \log(u_i)$, so we get
$$
\frac{D(u_i)}{u_i}\log(u_i)^{k_i} = \frac{1}{k_i+1}D(\log(u_i)^{k_i+1})
$$
and we can move such terms to elementary part.
If $u_i = 1$ or $v_i = 1$, then
the $\frac{D(u_i)}{u_i}\log(v_i)^{k_i}$ term vanishes.  In
particular this happens when $u_i$ or $v_i$ have zero at $p$.
Otherwise $u_i = 1 - v_i$ and we get polylog with argument
algebraic over $F$.
\end{proof}

\section{Extension by primitive}

We would like to investigate equalities for $I_k$.  To motivate
our approach
consider a vector space $V \subset K$ and its $k$-th tensor
power $V^{\otimes (k)}$.  On $V^{\otimes (k)}$ we
consider linear map $\Psi$ on simple tensor given by
$$g\otimes f_1 \dots \otimes f_{k-1} \mapsto Dg\prod_{l=1}^{k-1}f_l.$$
When $g = \log(1 - u)$ and $f_l = \log(u)$ we have
$$
\Psi(g\otimes f_1 \dots \otimes f_{k-1}) = DI_k.
$$
When tensor $s$ is symmetric then $\Psi(s)$ is a derivative
of element of $K$, so is negligible from the point of view
of integration.  So we are lead to study identities in
$V^{\otimes (k)}$ modulo symmetric tensors.  Below
we consider only case when $k=2$.

\begin{lemma}\label{tens-eq}
Assume $V$ is a vector space, $u, v, w_i, w_{i,j} \in V$,
$k_i, l_i \in {\mathbb Z}$, $k_i = l_i$ when $k_i < 0$ or $l_i < 0$,
$u = \sum_{i} k_iw_{i,j}$ for $l_j > 0$, $v = \sum_{i} l_iw_{i,j}$ for $k_j>0$,
$v - \sum_{i} l_iw_{i,j} = u - \sum_{i} k_iw_{i,j}$
when $k_j < 0$.
Put
$$
M_{i, j} = w_i\otimes w_j +  w_i\otimes w_{j,i} +  w_{i,j}\otimes w_j.
$$
Then
$$
(\sum_i k_iw_i + u)\otimes (\sum_j l_jw_j + v)
- \sum_{i,j} k_il_jM_{i, j} - u\otimes v$$
is a symmetric tensor.
\end{lemma}
\begin{proof}
We have
$$
(\sum_i k_iw_i + u)\otimes (\sum_j l_jw_j + v)
= (\sum_i k_iw_i)\otimes (\sum_j l_jw_j) $$
$$ 
+
u \otimes (\sum_j l_jw_j)
 + (\sum_i k_iw_i) \otimes v
 + u \otimes v.$$
From $M_{i, j}$ we get
$$
\sum k_il_jM_{i, j} = (\sum_i k_iw_i)\otimes (\sum_j l_jw_j)
+ \sum_i k_iw_i\otimes (\sum_j l_j w_{j,i})
$$
$$
+ \sum_j (\sum_i k_i w_{i,j})\otimes l_jw_j.
$$
So, sum in our claim is
$$
S_1 = \sum_j (u - \sum_i k_i w_{i,j})\otimes  l_jw_j
+ \sum_i k_i w_i\otimes(v - \sum_j l_j w_{j,i})
$$
It remains to show that $S_1$ is symmetric.

By assumption, when $l_j > 0$ we have $u = \sum_i k_iw_{i,j}$ so
$$
(u - \sum_i k_i w_{i,j}) \otimes l_jw_j = 0
$$
and
$$
\sum_{l_j > 0} (u - \sum_i k_i w_{i,j})\otimes  l_jw_j = 0.
$$
Similarly
$$
\sum_{k_i > 0} 
k_i w_i\otimes(v - \sum_j l_j w_{j,i}) = 0.
$$
By assumption, when $l_j < 0$ (so also $k_j < 0$) there are $t_j$ such that
$$
u - \sum_i k_i w_{i,j} = v - \sum_i l_iw_{i, j} = t_j.
$$
Using the equalities above we get
$$
S_1 = \sum_j (u - \sum_i k_i w_{i,j}) \otimes l_jw_j +
      \sum_i k_iw_i\otimes (v - \sum_j l_j w_{j,i})
$$
$$=
\sum_{l_j < 0}(u - \sum_i k_i w_{i,j}) \otimes l_jw_j +
      \sum_{k_i < 0}k_iw_i\otimes (v - \sum_j l_j w_{j,i})
$$
$$
= \sum_{l_j < 0}t_j\otimes l_jw_j +
  \sum_{k_i < 0}k_iw_i\otimes t_i =
 \sum_{l_j < 0}(t_j\otimes l_jw_j + l_jw_j\otimes t_j).
$$
where the last equality follows since negative $k_j$
are the same as negative $l_j$.  Since the result is
symmetric this ends the proof.
\end{proof}

\begin{lemma}\label{prep-ext}
Assume that $F$ is a differential field, $K$ is a differential
field algebraic over $F(\theta)$, 
$F$ and $K$ have the same constants,
$\theta$ is
a primitive over $F$.
Let $h_i \in K$ be a finite family.
Let $V$ be vector space over constants spanned by
logarithms with arguments algebraic over $F$.  Let $W$ be
vector space over constants spanned by logarithms of $h_i$ and $1 - h_i$.
Let $A$ be set of zeros and poles of $h_i$ and $1 - h_i$ in
algebraic closure of $F$.
There exist vector space over constants $X$, embedding $\iota$
from $V + W$ into $X\oplus V$, elements $\delta_a \in X$,
elements $u_i, v_i$ algebraic over $F$, elements $\beta_{a, b} \in V$
such that
$$
\iota(\log(h_i)) = \log(u_i) + \sum_{a\in A}\ord(h_i, a)\delta_a
$$
$$
\iota(\log(1 - h_i)) = \log(v_i) + \sum_{a\in A}\ord(1 -h_i, a)\delta_a
$$
when $\ord(1 - h_i, b) > 0$ we have
$$
\log(u_i) = \sum_{a\in A} \ord(h_i, a)\beta_{a, b}
$$
when $\ord(h_i, b) > 0$ we have
$$
\log(v_i) = \sum_{a\in A} \ord(1 - h_i, b)\beta_{a, b}
$$
when $\ord(h_i, b) < 0$ we have
$$
\log(u_i) + \sum_{a\in A} \ord(h_i, a)\beta_{a, b}
= \log(v_i) + \sum_{a\in A} \ord(1 - h_i, a)\beta_{a, b}.
$$
Moreover, for each $i$ either $v_i = 1 - u_i$ or $v_i = -u_i$
or $v_i = 1$ or $u_i = 1$.
\end{lemma}
\begin{proof}
Let $\bar F$ be algebraic closure of $F$.
Consider multiplicative
group $G$ generated by $h_i$, $1 - h_i$ and $\bar F_{*}$.  Choose place $p$ of
$\bar F K$.  Let $\psi_j$ be
generators of $G$ modulo $\bar F_{*}$, normalized so that
leading coefficient of Puiseaux expansion at $p$ is $1$.
By Lemma \ref{mult-free}
$G$ modulo $\bar F_{*}$ is a free abelian group.
As generators of free abelian group $\psi_j$ are
multiplicatively independent modulo $\bar F_{*}$.  Consequently,
by lemmas \ref{pole-indep0} and \ref{log-lin}
$\log(\psi_j)$ are algebraically independent over $F$.

There are integers $m_{i, j}$ and $n_{i, j}$ and elements
$u_i, v_i \in \bar F$
such that
$$
h_i = u_i\prod \psi_j^{m_{i,j}},
$$
$$
(1 - h_i) = v_i\prod \psi_j^{n_{i,j}}.
$$
Note that when both $h_i$ and $(1 - h_i)$ have order $0$ at $p$, then
$v_i = 1 - u_i$.  When one of $h_i$ and $(1 - h_i)$ have negative
order at $p$, then also the second have negative order
and $v_i = u_i$.  When $h_i$ have positive order at $p$, then
$v_i = 1$, When $1 - h_i$ have positive order at $p$, then
$u_i = 1$.  This covers all cases, so the last claim holds.

Let $W_0$ be vector space over constants spanned by
$\{\log(\psi_j)\}$ and
$\bar W$ be vector space over constants spanned by
$W_0$ and $V$. 
Let $\bar A$ be set of zeros and poles of $\psi_j$-s over
$\bar F$.  We take as $X$ vector space over constants with
basis $\delta_a$, $a \in \bar A$ (actually, for final result we
only need $a \in A$, but for the proof we need larger $X$).

We define $\iota$ on $\bar W$ by the formula
$$
\iota(\log(\psi_j)) = \sum_{a\in \bar A}\ord(\psi_j, a)\delta_a,
$$
on $W_0$ and as identity on $V$.
Clearly on $V$
$\iota$ is well-defined and is injective.
Since $\psi_j$ are multiplicatively independent modulo algebraic
closure of $F$ $\log(\psi_j)$ are linearly independent modulo $V$
so $\bar W$ is a direct sum of $W_0$ and $V$.
In particular
it follows that $\iota$ is well defined.  Also, element in kernel
of $\iota$ has form
$$
\sum c_j\log(\psi_j)
$$
where $c_j$ are constants such that for each $a \in \bar A$
we have
$$
\sum c_j \ord(\psi_j, a) = 0.
$$
Let $e_l$ be basis of vector space over rational numbers spanned
by $c_j$-s.  We can take $e_l$ such that all $c_j$ have integer
coordinates.  That is
$$
c_j = \sum_l q_{j,l}e_l.
$$
Then, for each $a \in \bar A$
$$
\sum_l e_l(\sum_j q_{j,l}\ord(\psi_j, a)) = 0.
$$
Since $e_l$ are linearly independent over rationals that means
that for each $l$ and $a$
$$
\sum_j q_{j,l}\ord(\psi_j, a) = 0.
$$
Since function without zeros and poles is in $\bar F_{*}$ and
$\psi_j$ are multiplicatively independent over $\bar F_{*}$
equality above means that $q_{j,l} = 0$ so also $c_j = 0$.
This means that kernel of $\iota$ is trivial, that is
$\iota$ is injective.

We have
$$
\iota(\log(h_i)) = \iota(\log(u_i) + \sum m_{i,j}\log(\psi_j))
$$
$$
= \log(u_i) + \sum m_{i,j}\sum_a\ord(\psi_j, a)\delta_a
= \log(u_i) + \sum_a\ord(h_i, a)\delta_a
$$
and similarly for $\log(1 - h_i)$.
This shows first two conditions in conclusion of the lemma.

We need to define $\beta_{a, b}$ and show that they have
required properties.  We first define $\alpha_{j, a}$ as
leading coefficient of Puiseaux expansion of $\psi_j$
at $a$ and put $\gamma_{j, a} = -\log(\alpha_{j, a})$.

When $a$ is a zero of $h_i$ at $a$ we have $1 - h_i = 1$ so
$$
1 = v_i\prod \alpha_{j, a}^{n_{i, j}}
$$
and
$$
\log(v_i) = \sum_j n_{i, j}\gamma_{j, a}.
$$
When $a$ is z zero of $1 - h_i$ at $a$ we have
$h_i = 1 - (1 - h_i) = 1$ so
$$
1 = u_i\prod \alpha_{j, a}^{m_{i, j}}
$$
and
$$
\log(u_i) = \sum_j m_{i, j}\gamma_{j, a}.
$$
When $a$ is a pole of $h_i$, then $a$ is also a pole of $1 - h_i$ and
we have
$$
u_i\prod \alpha_{j, a}^{m_{i,j}} = -w_1\prod \alpha_{j, a}^{n_{i,j}}
$$
so
$$
\log(u_i) - \sum_j m_{i, j}\gamma_{j, a} =
\log(w_i) - \sum_j n_{i, j}\gamma_{j, a}.
$$
For fixed $b$ we can view $\gamma_{j, b}$ as values of
a linear operator $T_b$ defined on $W_0$ by the formula:
$$
T_b(\sum c_j\log(\psi_j)) = \sum c_j\gamma_{j, b}.
$$
Since $\iota$ is injective on $W_0$ and takes values in $X$
we can treat $T_b$ as operator defined on a subspace of $X$ and
extend it to linear operator $\tilde T_b$ defined on whole $X$.
We put
$\beta_{a, b} = T_b(\delta_a)$.  Since
$\iota(\log(\psi_j)) = \sum_a \ord(\psi_j, b)\delta_a$ we have
$\gamma_{j, b} = \sum_a \ord(\psi_j, a)\beta_{a, b}$.
Next, when $\ord(1 - h_i, b) > 0$ (that is $b$ is zero
of $1 - h_i$) we have
$$
\log(u_i) = \sum_j m_{i, j}\gamma_{j, b} =
\sum_j m_{i, j}\sum_a \ord(\psi_j, a)\beta_{a, b} 
$$
$$
= \sum_a(\sum_j m_{i, j}\ord(\psi_j, a)) \beta_{a, b}
= \sum_a\ord(h_i, a)) \beta_{a, b}.
$$
Similarly
when $\ord(h_i, b) > 0$ we have
$$
\log(v_i) = \sum_a\ord(1 - h_i, a)) \beta_{a, b}
$$
and when $\ord(h_i, b) < 0$ we have
$$
\log(u_i) - \sum_a\ord(h_i, a)) \beta_{a, b}
= \log(v_i) - \sum_a\ord(1 - h_i, a)) \beta_{a, b}
$$
so $\beta_{a, b}$ satisfy conclusion of the lemma.
\end{proof}

Now we can prove Theorem \ref{prim-case}

\begin{proof}
When $\theta$ is algebraic over $F$ there is nothing to
prove.  So we may assume that $\theta$ is transcendental
over $F$.
Let $\bar F$ be algebraic closure of $F$.  We have
$$
f = g_F + DE + \sum d_i\frac{D(1 - h_i)}{1 - h_i}\log(h_i)
$$
where $g_F$ is sum of dilog terms with arguments in $\bar F$,
$E$ denotes elementary parts and $h_i \in K - \bar F$.  Since
$$
D(\log(1 - h_i)\log(h_i)) = \frac{D(1- h_i)}{1 - h_i}\log(h_i)
+ \frac{D(h_i)}{h_i}\log(1 - h_i)
$$
by changing elementary part we may assume that dilog terms are
antisymmetric, that is
\begin{equation}\label{inteq}
f = g_F + DE + \sum d_i\left(\frac{D(1 - h_i)}{1 - h_i}\log(h_i) - 
\frac{D(h_i)}{h_i}\log(1 - h_i)\right).
\end{equation}
Without loss of generality we may assume that arguments of
logarithms in elementary part are either in $F$ or appear
among $h_i$.  We now use Lemma \ref{prep-ext} obtaining
$u_i$, $v_i$, etc. with properties stated in the Lemma.
Like in the proof of the Lemma we introduce space $\bar W$
containing $\log(h_i), \log(1 - h_i), \log(u_i), \log(v_i)$.
On $\bar W \otimes \bar W$ we consider mapping $\Psi$ given
by the formula
$$
\Psi(t \otimes s) = (Dt)s.
$$
We have
$$
\sum d_i\left(\frac{D(1 - h_i)}{1 - h_i}\log(h_i) -
\frac{D(h_i)}{h_i}\log(1 - h_i)\right)
$$
$$
= \sum d_i
\Psi\left(
\log(1 - h_i) \otimes \log(h_i) - \log(h_i) \otimes \log(1 - h_i)
\right).
$$
Now, to prove the theorem it is enough to show that
$$
S_1 =
\sum d_i
(
\log(1 - h_i) \otimes \log(h_i) - \log(h_i) \otimes \log(1 - h_i)
$$
$$
- \log(v_i) \otimes \log(u_i) + \log(u_i)\otimes \log(v_i))
$$
is a symmetric tensor.  Namely, on symmetric tensor
$\Psi(t\otimes s + s\otimes t) = D(st)$ so values of $\Psi$
on symmetric tensors are derivatives of elementary functions.
So by changing elementary part we can replace
$$
\sum d_i\left(\frac{D(1 - h_i)}{1 - h_i}\log(h_i) -
\frac{D(h_i)}{h_i}\log(1 - h_i)\right)
$$
by
$$
\sum d_i\left(\frac{D(v_i)}{v_i}\log(u_i) -
\frac{D(u_i)}{u_i}\log(v_i)\right).
$$
By the last claim of Lemma \ref{prep-ext} we have four possibilities
for $u_i$ and $v_i$.  When $u_i = 1 - v_i$ term of the sum above
is just dilog term with argument algebraic over $F$.  When
$v_i = -u_i$, then $\log(v_i) = \log(u_i)$ and corresponding term
is zero.  Also, when $u_i = 1$ or $v_i = 1$,
then corresponding term is zero.  So all dilog terms have
arguments algebraic over $F$ as claimed.

It remains to show that $S_1$ is a symmetric tensor.  Since $\iota$
is an embedding, it is enough to show that $(\iota\otimes\iota)(S_1)$ is
symmetric.   By Lemma \ref{tens-eq} putting
$$
M_{a, b} = \delta_a\otimes \delta_b + \delta_a\otimes\beta_{b, a}
+ \beta_{a, b}\otimes \delta_b.
$$
for each $i$ we have
$$
(\iota\otimes\iota)(\log(1 - h_i)\otimes\log(h_i) - \log(v_i)\otimes\log(u_i))
$$
$$ =
\sum_{a, b} \ord(1 - h_i, a)\ord(h_i, b)M_{a, b}
$$
modulo symmetric tensors.  Consequently, modulo symmetric tensors
$$
(\iota\otimes\iota)(S_1) = \sum_i d_i\sum_{a, b}\ord(1 - h_i, a)\ord(h_i, b)
(M_{a, b} - M_{b, a}).
$$
So, it is enough to show that the last sum equals $0$.
Consider projection $\pi$ from $V \oplus X$ onto space $X$.
We have
$$
(\pi\otimes \pi)(M_{a, b} - M_{b, a}) = 
\delta_a\otimes \delta_b - \delta_b\otimes \delta_a
$$
so projections of $M_{a, b} - M_{b, a}$ give linearly independent
antisymmetric tensors.  So it is enough to show that
$(\pi\otimes \pi)(\iota\otimes\iota)(S_1)$ is $0$, because then coefficients
of $M_{a, b} - M_{b, a}$ must be zero.  However, $S_1$ is
in $(V+W)\otimes(V+W)$ so we can give more explicit formula for
$(\pi\otimes \pi)(\iota\otimes\iota)(S_1)$.  Let $W_0$ be space spanned by
$\log(\psi_j)$ from the proof of Lemma \ref{prep-ext}.  We
showed that $V+W = V\oplus W_0$ and by formula for $\iota$
we see that $\pi\iota$ is just projection $\chi$ from $V+W$ onto $W_0$
followed by embedding from $W_0$ into $X$.  So, it is enough
to show that $S_2 = (\chi\otimes\chi)S_1$ equals $0$.
$\chi$ maps $\log(u_i)$ and $\log(v_i)$
to $0$ while $\log(h_i)$ and $\log(1 - h_i)$ are mapped
to linear combinations of $\log(\psi_j)$.  Note that
$$
\log(1 - h_i)\otimes\log(h_i) - \log(h_i)\otimes\log(1 - h_i)
$$
is mapped to an antisymmetric tensor.
So $S_2$ is an antisymmetric tensor.
Consider $\Psi(S_2)$.  We have
$$
S_2 = \sum c_{k,j}\log(\psi_k)\otimes\log(\psi_j),
$$
$$
\Psi(S_2) = \sum c_{k,j}D(\log(\psi_k))\log(\psi_j)
$$
where $c_{k,j}$ are constants and $c_{k,j} = -c_{j, k}$.
Note that $\log(h_i) - \chi(\log(h_i)) = \log(u_i)$
and similarly for $1 - h_i$, so $\Psi(S_1) = \Psi(S_2) + \Psi(R)$
where $R$ contains terms containing $\log(u_i)$ and $\log(v_i)$.
Let $L$ be $F$ extended by $u_i$, $v_i$, $\log(u_i)$ and $\log(v_i)$.
Now, expand equation (\ref{inteq}) as polynomial in $\log(\psi_j)$
and consider terms linear in $\log(\psi_j)$.  Elementary
part is a polynomial of degree $2$ with coefficients in $KL$.
Since terms of second order in (\ref{inteq}) are $0$ coefficients
of second order elementary terms are constants, that is
$$
E = \sum b_{k, j}\log(\psi_k)\log(\psi_j) + \sum s_j\log(\psi_k) + s_0
$$
where $s_j \in KL$, and $b_{k, j}$ are constant which we can
assume to be symmetric, that is $b_{k, j} = b_{j, k}$.
So coefficient of
$\log(\psi_j)$ in (\ref{inteq}) is
$$
\sum_k (c_{k,j} + 2b_{k, j})D(\log(\psi_k)) + D(a_j)
$$
where $a_j \in KL$ is a sum of a linear combination
of logarithms with arguments algebraic over $F$ coming from $\Psi(S_1)$
and $s_j$.  Using Lemma \ref{pole-indep0} we see that
$c_{k,j} + 2b_{k, j} = 0$.  Since $c_{k,j}$ are antisymmetric
and $b_{k, j}$ are symmetric it follows that $c_{k,j} = 0$.
However, this means that $S_2 = 0$, which ends the proof.
\end{proof}

\section{Further remarks}

Our lemmas \ref{tens-eq} and \ref{prep-ext} correspond to
Proposition 2 in \cite{Bad94} and \cite{Bad06}.  When
in Lemma \ref{prep-ext} field
$K$ is purely transcendental over $F$ we could use
infinite place only for normalization (do not include it in $A$),
take $\delta_a = \log(\theta - a)$ and $\beta_{a, b} = \log(a - b)$.
Baddoura instead of our $M_{i, j}$ uses Spence function
evaluated at $(\theta - b)/\theta - 0)$.  Both expression
gave the same main part.  Spence function
has builtin antisymmetry, we work modulo symmetric tensors.
Also, our $M_{i, j}$ omits lowest order term present in
Spence function.  This considerably simplifies handling
of lowest order terms in our proof.  Baddoura uses different
condition on $\beta_{a, b}$, which is equivalent to our condition
in transcendental case thanks to symmetry of $\beta_{a, b}$,
but in general we are unable to find symmetric $\beta_{a, b}$.

Our lemmas \ref{tens-eq} and \ref{prep-ext} are harder to use
than Baddoura's identity: we do not know if there is actual
function associated to an element of $X$.

Tensor products seem to be standard tool used for studying
polylogarithms (we learned about it from \cite{Duhr:poly})
but seem to be new in context of symbolic integration.

Our Lemma \ref{log-int} is simple and should be well-known
(it could be easily extracted from proof of Liouville
theorem).  It replaces longish arguments used in
\cite{Bad94} and \cite{Bad06}.

We hope to extend Theorem \ref{prim-case} to polylogarithms
of arbitrary order.  For purpose of symbolic integration
one would like to have more precise information.  In particular,
we would like to have method to find arguments of polylogarithms
needed in given integral.  Here, our current result have significant
weakness: in principle we need arbitrary algebraic extension
to find arguments of polylogarithms.  It seems reasonable that
we only need new algebraic constants and that we can find
arguments of polylogarithms in base field extended by constants.


\begin{thebibliography}{10}

\bibitem{Bad94}
{J. Baddoura,
Integration in finite terms with elementary functions and dilogarithms,
Phd thesis MIT 1994.}

\bibitem{Bad06}
{J. Baddoura,
Integration in Finite Terms with Elementary Functions and Dilogarithms,
Journal of Symbolic Computation 41 (2006), 909-942.}

\bibitem{Bad11}
{J. Baddoura,
A note on symbolic integration with polylogarithms,
Mediterr. J. Math. 8 (2011), 229-241.}

\bibitem{Bro:ele}
{M. Bronstein,
Integration of elementary functions,
J. of Symb. Comp. 9 (1990), 117--173.}

\bibitem{Bro:str}
{M. Bronstein,
Structure theorems for parallel integration,
J. of Symb. Comp. 42 (2007), 757--769.}

\bibitem{Chev}
{C. Chevalley,
Introduction to the theory of algebraic functions of one variable,
AMS 1951.}

\bibitem{Duhr:poly}
{C. Duhr, H. Gangl, J. R. Rhodes,
From polygons and polygon symbols to polylogarithmic functions,
\verb|https://arxiv.org/abs/1110.0458|}

\bibitem{KaSri}
{Y. Kaur, V. R. Srinivasan,
Integration in finite terms with dilogarithmic integrals,
logarithmic integrals and error functions,
J. of Symb. Comp. 94 (2019), 283--302.}

\bibitem{Kummer}
{E.E. Kummer,
\"Uber die Transcendenten, welche aus wieder holten Integrationen rationaler 
Funktionen entstehen,
JRAM (Crelle) 21 (1840), 74--90.}

\bibitem{Ros}
{M. Rosenlicht,
On Liouville's theory of elementary functions,
Pacific J. of Math. 65 (1976), 485--492.}

\bibitem{SSC}
{M. F. Singer, B. D. Saunders, B. F. Caviness,
An extension of Liouville's theorem of integration
in finite terms,
SIAM J. Comput. 14 (1985), 966--990.}

\end{thebibliography}
\end{document}